\documentclass{amsart}
\usepackage{amsfonts}
\usepackage[latin1]{inputenc}

\newtheorem{thm}{Theorem}[section]

\newtheorem{prop}[thm]{Proposition}
\newtheorem{defn}[thm]{Definition}

\def\CC{{\bf C}}
\def\ll{{ L}}
\begin{document}
	\title[Central extensions of some nilpotent Leibniz algebras.] {Central extensions of null-filiform and naturally graded filiform non-Lie Leibniz algebras.}

\author{J.K. Adashev}
\address{[J.K. Adashev] Institute of Mathematics, National University of Uzbekistan, Tashkent, 100125, Uzbekistan}
\email{adashevjq@mail.ru}
\author{L.M. Camacho}
\address{[L.M. Camacho] Dpto. Matem\'{a}tica Aplicada I. Universidad de Sevilla. Avda. Reina Mercedes, s/n. 41012 Sevilla. (Spain)}
\email{lcamacho@us.es}
\author{B.A. Omirov}
\address{[B.A. Omirov] Institute of Mathematics, National University of Uzbekistan, Tashkent, 100125, Uzbekistan}
\email{omirovb@mail.ru}
	
\maketitle
		
\begin{abstract}
In this paper we describe central extensions of some nilpotent Leibniz algebras. Namely, central extensions of the Leibniz algebra with maximal index of nilpotency are classified. Moreover, non-split central extensions of naturally graded filiform non-Lie Leibniz algebras are described up to isomorphism. It is shown that $k$-dimensional central extensions ($k\geq 5$) of these algebras are split.
\end{abstract}

	
	

\section{Introduction}

Central extensions play an important role in quantum mechanics: one of the earlier encounters is by means of Wigner's theorem which states that a symmetry of a quantum mechanical system determines a (anti-) unitary transformation of a Hilbert space.

Another area of physics where one encounters central extensions is the quantum theory of conserved currents of a Lagrangian. These currents span an algebra which is closely
related to so called affine Kac-Moody algebras, which are the universal central extension of loop algebras.

Central extensions are needed in physics, because the symmetry group of a quantized system usually is a central extension of the classical symmetry group, and in the same way the corresponding symmetry Lie algebra of the quantum system is, in general, a central extension of the classical symmetry algebra. Kac-Moody algebras have been conjectured to be a symmetry groups of a unified superstring theory. The centrally extended Lie algebras play a dominant role in quantum field theory, particularly in conformal field theory, string theory and in $M$-theory.

In the theory of Lie groups, Lie algebras and their representations, a Lie algebra extension is an enlargement of a given Lie algebra $g$ by another Lie algebra $h$. Extensions arise in several ways. There is a trivial extension obtained by taking a direct sum of two Lie algebras. Other types are split extension and central extension. Extensions may arise naturally, for instance, when forming a Lie algebra from projective group representations.

A central extension and an extension by a derivation, one obtains a Lie algebra which is isomorphic with a non-twisted affine Kac-Moody algebra (\cite{bauerle}, Chapter 19). Using the centrally extended loop algebra one may construct a current algebra in two spacetime dimensions. The Virasoro algebra is the universal central extension of the Witt algebra, the Heisenberg algebra is the central extension of a commutative Lie algebra (\cite{bauerle}, Chapter 18), \cite{cahen}, \cite{szy}.

During last 20 years Leibniz algebras are studied mainly to extend the classical results from Lie algebras.
From structural point of view some classical analogues which are true for nilpotent Lie algebras are obtained.
In particular, papers devoted to the study of naturally graded nilpotent Leibniz algebras are \cite{Omirov1}, \cite{Comm-2-fil}, \cite{Comm-p-fil}, \cite{JSC},  \cite{gomez-omirov} and many others.

In the present paper we consider a notion of a central extension of Leibniz algebras, which is similar to the notion of a central extension of Lie algebras \cite{sk}. Here we consider central extensions of some nilpotent Leibniz algebras, such as null-filiform Leibniz algebra and naturally graded non-Lie filiform Leibniz algebras.
In fact, the method of central extensions of Lie algebras is adapted for Leibniz algebras in \cite{rakh}.
In works \cite{rakh}-\cite{rakh1} one-dimensional central extensions of the null-filiform Leibniz algebra and the naturally graded Lie algebra are described. The difficulties of the description of one-dimensional central extensions of a naturally graded Lie algebra show that the further progress in description of filiform Lie algebras by using central extensions is too complicated. That is why in our study we restrict to non-Lie Leibniz algebras.

The main results of this paper consists of the classification of central extensions of null-filiform Leibniz algebras and naturally graded filiform non-Lie Leibniz algebras.

Throughout the paper all vector spaces and algebras are finite-dimensional over the field of complex numbers. Moreover, in the table of multiplication of an algebra the omitted products are assumed to be zero.

\section{Preliminaries}

Let $\mathbb{F}$ be an algebraically closed field of characteristic 0.
\begin{defn}\cite{loday}
	An algebra  $L$ over a field $\mathbb{F}$ is called a Leibniz algebra if for any elements $x, y, z \in L$
	the so-called Leibniz identity is satisfied:
	$$[x, [y, z]] = [[x, y], z] - [[x, z], y],$$
	where $[-,-]$ is a multiplication in $L.$
\end{defn}

Let $L$ be a finite-dimensional Leibniz algebra. For this algebra we define sequence:
$$L^{1}=L, \quad L^{k+1}=[L^{k}, L^{1}],  \quad k \geq 1.$$

A Leibniz algebra $L$ is called nilpotent if there exists
$s\in N$ such that $L^s = \{0\}.$ The minimal number $s$ possessing of such property  is called index of nilpotency of the algebra $L$.

\begin{defn} \cite{Omirov1} A Leibniz algebra $L$ of dimension $n$ is called null-filiform if $dimL^i = (n + 1) - i,  \ 1 \leq i \leq n + 1.$
\end{defn}

Evidently, by definition algebra being null-filiform is equivalent that it has maximal possible index of nilpotency.

\begin{thm}\cite{Omirov1} In arbitrary $n$-dimensional null-filiform Leibniz algebra there exists a basis $\{e_1, e_2,\dots, e_n\}$ such that the table of multiplication of the algebra is in the following form:
	$$NF_n: \ [e_i,e_1]=
	e_{i+1}, \ \  1\leq i\leq n-1.$$
\end{thm}

\begin{defn}\cite{Omirov1} A Leibniz algebra $L$ is called filiform if $dim L^i = n-i, \ 2 \leq i \leq n.$
\end{defn}

Note that the nilindex of a filiform Leibniz algebras coincides with its dimension.

Let $L$ be a Leibniz algebra of nilindex $s.$ We put
$L_i=L^i/L^{i+1},  \ \ 1\leq i\leq s-1.$ Then we obtain a graded algebra $GrL=L_1\oplus L_2\oplus\cdots\oplus L_{s-1},$ where $[L_i,L_j]\subseteq L_{i+j}.$ An algebra $L$ is called a naturally graded algebra if $L\cong GrL.$

In the following theorem we summarize the results of the works \cite{Omirov1}, \cite{Ve}.

\begin{thm} An arbitrary complex $n$-dimensional naturally graded filiform
Leibniz algebra is isomorphic to one of the following pairwise non-isomorphic algebras:
$$F^1_n: \begin{array}{ll}
[e_1,e_1]=e_{3},& [e_i,e_1]=e_{i+1},  \  2\leq i \leq {n-1},
\end{array}$$
$$F^2_n: \begin{array}{ll}
[e_1,e_1]=e_{3}, &  [e_i,e_1]=e_{i+1},  \  3\leq i \leq {n-1},
\end{array}$$
$$F^3_n: \begin{array}{lll} [e_i,e_1]=-[e_1,e_i]=e_{i+1}, &
2\leq i \leq {n-1}\\[1mm]
[e_i,e_{n+1-i}]=-[e_{n+1-i},e_i]=\alpha (-1)^{i+1}e_n, & 2\leq i\leq
n-1,\end{array}$$
where $\alpha\in\{0,1\}$ for even $n$ and $\alpha=0$ for odd $n.$
\end{thm}

It is known that the set of filiform Leibniz algebras is decomposed into three disjoint classes, two families of non-Lie Leibniz algebras which arise from the algebras $F_n^1$ and $F_n^2$ and one family of algebras (which include filiform Lie algebras) arises from the Lie algebra $F_n^3$ \cite{gomez-omirov}.

Since further we consider only filiform non-Lie Leibniz algebras we omit the third family of algebras.

\begin{thm} Any $(n+1)$-dimensional complex non-Lie filiform Leibniz algebra, whose naturally graded algebra is not a Lie algebra, belongs to one of the following two classes:
$$\begin{array}{l}
F_n^1(\alpha_4,\alpha_5, \dots, \alpha_n, \theta
):\left\{\begin{array}{ll}
  [e_1,e_1]=e_{3}, &  \\[1mm]
  [e_i,e_1]=e_{i+1}, & 2\leq i \leq {n-1}, \\[1mm]
  [e_1,e_2]=\sum\limits_{k=4}^{n-1}\alpha_{k}e_k+\theta e_n, &  \\[1mm]
  [e_i,e_2]=\sum\limits_{k=i+2}^n\alpha_{k+1-i}e_k, &  \ 2\leq i \leq
{n-2}, \\
\end{array}
  \right.\\[15mm]
 F_n^2(\beta_4,\beta_5, \dots, \beta_n, \gamma ):\left\{
\begin{array}{ll}
  [e_1,e_1]=e_{3}, &  \\[1mm]
  [e_i,e_1]=e_{i+1}, & 3\leq i \leq {n-1}, \\[1mm]
  [e_1,e_2]=\sum\limits_{k=4}^{n}\beta_{k}e_k, &\\[1mm]
  [e_2,e_2]= \gamma e_n,   \\[1mm]
  [e_i,e_2]=\sum\limits_{k=i+2}^n\beta_{k+1-i}e_k, &  \ 3\leq i \leq
{n-2}. \\
\end{array}
  \right.
\end{array}$$
\end{thm}

\begin{defn} An algebra $L$ is called quasi-filiform if $L^{n-2}\neq\{0\}$ and
	$ L^{n-1}=\{0\},$ where $dimL=n.$
\end{defn}


Let $x$ be a nilpotent element of the set $\ll\setminus \ll^2$. For the
nilpotent operator of right multiplication $R_{x}$ we define a decreasing
sequence $C(x)=(n_{1}, n_{2}, \dots, n_{k})$, which consists of the dimensions
of Jordan blocks of the operator $R_{x}$. In the set of such sequences we
consider the lexicographic order, that is, $C(x)=(n_{1}, n_{2},\dots, n_{k})$
$\leq $ $C(y)=(m_{1}, m_{2},\dots, m_{s})$ if and only if there exists
$i\in \mathbb{N}$ such that $n_{j}=m_{j}$ for any  $j < i$ and $n_{i} < m_{i}$.

\begin{defn}
The sequence $C(\ll)=\mbox{max C(x)}_{x\in \ll\backslash \ll^2}$ is called
characteristic sequence of the algebra $\ll$.
\end{defn}

\begin{defn}
A quasi-filiform non Lie Leibniz algebra $\ll$ is called algebra of
type I, if there exists an element $x\in \ll\backslash
\ll^2$ such that the operator $R_x$ has the form:
$\left(\begin{array}{ll}
J_{n-2}&0\\
0&J_{2}
\end{array}\right).$
\end{defn}

Combining the results of the papers \cite{Comm-2-fil} and \cite{JSC} we have the following theorem.

\begin{thm}
Let $\ll$  be a naturally graded quasi-filiform Leibniz algebra of type $I$. Then
it is isomorphic to the one of the following pairwise non-isomorphic
algebras:
$$NF_{n-2}\oplus \mathbb{C}^2; \ \ F_{n-1}^1\oplus \mathbb{C};$$
$$\small\begin{array}{ll}

L_n^1: \left\{\begin{array}{ll}
[e_i,e_1]=e_{i+1},& 1\leq i\leq n-3,\\{}
[e_1,e_{n-1}]=e_n+e_2,& \\{}
[e_i,e_{n-1}]=e_{i+1},& 2\leq i\leq n-3.
\end{array}\right.&
L_n^2: \left\{\begin{array}{ll}
[e_i,e_1]=e_{i+1},& 1\leq i\leq n-3,\\{}
[e_1,e_{n-1}]=e_n.&
\end{array}\right.
\end{array}$$
$$\small\begin{array}{ll}
	\ll^{1,\lambda}_n: \left\{
	\begin{array}{l}
		[e_i,e_1]=e_{i+1},\ 1\leq i \leq n-3\\{} [e_{n-1},e_1]=e_{n},\\{}
		[e_1,e_{n-1}]=\lambda e_{n},\ \lambda\in \CC
	\end{array}\right.&
	\ll^{2,\lambda}_n: \left\{
	\begin{array}{l}
		[e_i,e_1]=e_{i+1},\ 1\leq i \leq n-3\\{} [e_{n-1},e_1]=e_{n},\\{}
		[e_1,e_{n-1}]=\lambda e_{n},\ \lambda\in \{0,1\}\\{}
		[e_{n-1},e_{n-1}]=e_n
	\end{array}\right.\\[7mm]
	
	\ll^{3,\lambda}_n: \left\{
	\begin{array}{l}
		[e_i,e_1]=e_{i+1},\ 1\leq i \leq n-3\\{}
		[e_{n-1},e_1]=e_{n}+e_2,\\{} [e_1,e_{n-1}]=\lambda e_{n},\
		\lambda\in \{-1,0,1\}
	\end{array}\right.&
	
	\ll^{4,\lambda}_n: \left\{
	\begin{array}{l}
		[e_i,e_1]=e_{i+1},\ 1\leq i \leq n-3\\{}
		[e_{n-1},e_1]=e_{n}+e_2,\\{} [e_{n-1},e_{n-1}]=\lambda e_n,\
		\lambda\neq 0
	\end{array}\right.\\[7mm]
	
	\ll^{5,\lambda,\mu}_n: \left\{
	\begin{array}{l}
		[e_i,e_1]=e_{i+1},\ 1\leq i \leq n-3\\{} [e_{n-1},e_1]=e_{n}+e_2,\\{}
		[e_1,e_{n-1}]=\lambda e_{n},\ (\lambda,\mu)=(1,1)\ or\ (2,4)\\{}
		[e_{n-1},e_{n-1}]=\mu e_n,
	\end{array}\right.&
	\ll^{6}_n: \left\{
	\begin{array}{l}
		[e_i,e_1]=e_{i+1},\ 1\leq i \leq n-3\\{} [e_{n-1},e_1]=e_{n},\
		\\{} [e_1,e_{n-1}]=- e_{n},\\{} [e_{n-1},e_{n-1}]=e_2,\\{}
		[e_{n-1},e_n]=e_3.
	\end{array}\right.
\end{array}$$

\end{thm}

\

{\bf Central extensions of Leibniz algebras.}
It is remarkable that any nilpotent Leibniz algebra has non-trivial center. Let $L$ be a Leibniz algebra over a field $\mathbb{F}$ and $Z(L)$ be its center.

{\it The central extension of a Leibniz algebra $L$ by $V$} is called a short exact sequence of Leibniz algebras

$$0\longrightarrow V\mathop{\longrightarrow}\limits^{\varepsilon} \widehat{L} \mathop{\longrightarrow}\limits^{\lambda} L\longrightarrow0,$$
where $Im \varepsilon= Ker \lambda$ and $V$ is the center of the algebra $\widehat{L}$.

Two central extensions
$$0\longrightarrow V\mathop{\longrightarrow}\limits^{\epsilon} \widehat{L} \mathop{\longrightarrow}\limits^{\lambda}L\longrightarrow0
\ \ \mbox{and} \ \  0\longrightarrow V
\mathop{\longrightarrow}\limits^{\epsilon^\prime}
 \widehat{L}^\prime
\mathop{\longrightarrow}\limits^{\lambda^\prime}
L\longrightarrow0$$ are called {\it equivalent} if there exists an isomorphism $\varphi :\widehat{L} \longrightarrow
\widehat{L}^\prime$ such that $\varphi\circ\epsilon=\epsilon^\prime,  \
\lambda^\prime\circ\varphi=\lambda.$

A bilinear map $\theta: L\times L\longrightarrow V$ satisfying an identity
$$\theta(x, [y, z])=\theta([x,y],z)-\theta([x,z],y),$$
for any elements $x,y,z\in L$ is called {\it central 2-cocycle}.

We denote by $ZL^2(L,V)$ the set of all central $2$-cocycles from $L$ to $V$. If for linear map $\varphi:L\longrightarrow V$ we have $\theta(x,y)=\varphi([x,y]),$ then $\theta$ is called {\it $2$-coboundary.} We denote by $BL^2(L,V)$ the set of all $2$-coboundaries from $L$ to $V$. The quotient space $HL^2(L,V):=ZL^2(L,V)/BL^2(L,V)$ is called {\it the 2-nd group of cohomology.}
If $\theta_1-\theta_2$ is coboundary, then $\theta_1$ and $\theta_2$ are {\it called cohomological.}

For a given $2$-cocycle $\theta$ on $L$ we construct the associated central extension $L_{\theta}$. By definition it is a vector space $L_{\theta}=L\oplus V$ with equipped multiplication:
$$[x+u, y+v]=[x,y]_L +\theta(x,y), \eqno(1)$$
where $[\cdot , \cdot]_L$ is the bracket on $L$ and $x,y\in L,  \  \ u,v\in V.$

\section{Classification of $k$-dimensional central extensions of null-filiform Leibniz algebras}

In this section, we focus on $k$-dimensional central extensions of null-filiform Leibniz algebra. First, we compute a basis of $HL^2(NF_n,V).$

\begin{prop}\label{cocyclenul}
	Let $L$ be the null-filiform Leibniz algebra and let $V=<x_1,x_2,\dots,x_k>$ be an abelian algebra. Then
	\begin{itemize}
	
	\item A basis of $ZL^2(NF_n,V)$ is formed by the following cocycles $$\theta_{i,j}(e_i,e_1)=x_j, \ \  1\leq i\leq n, \ \
	1\leq j\leq k$$
	
	\item A basis of $BL^2(NF_n,V)$ is formed by the following coboundaries $$\theta_{i,j}(e_i,e_1)=x_j, \ \  1\leq i\leq n-1, \ \
	1\leq j\leq k$$
	
	\item A basis of $HL^2(NF_n, V)$ is formed by the following cocycles $$\theta_{j}(e_n,e_1)=x_j,  \ \ 1\leq
	j\leq k$$
		\end{itemize}
\end{prop}
\begin{proof}
The proof follows directly from the definition of a cocycle.
\end{proof}

\begin{thm}
A central extension of an $n$-dimensional complex null-filiform Leibniz algebra is isomorphic to one of the following non-isomorphic algebras:
	$$ NF_{n}\oplus \mathbb{C}^{k}, \quad NF_{n+1}\oplus \mathbb{C}^{k-1}.$$
\end{thm}

\begin{proof} Let  $\{e_1,e_3,\dots,e_n\}$ be a basis of $NF_n$ and $\{x_1,\dots,x_k\}$ be a basis of $V.$
From the definition of the product on central extension $(1)$ we conclude that $\{e_1,e_3,\dots,e_n,
x_1,x_2,\dots,x_k\}$ forms a basis of the algebra $L=NF_n\oplus V.$

Applying Proposition \ref{cocyclenul}, we get that the table of multiplication of the central extension of $NF_n$ has the following form:
$$\begin{array}{ll}
[e_i,e_1]=e_{i+1},  & 1\leq i\leq n-1,\\[1mm]
[e_n,e_1]=\sum\limits_{i=1}^k\alpha_i x_i,&
\end{array}$$

If $\alpha_i=0$ for all $1\leq i\leq k,$  then we obtain the algebra $NF_n\oplus \mathbb{C}^k.$

If there exists $\alpha_i\neq0$ for some $i$, then setting $e_{n+1}=\sum\limits_{i=1}^k\alpha_i x_i$ we get the algebra $NF_{n+1}\oplus \mathbb{C}^{k-1}.$
\end{proof}

\section{Classification of $k$-dimensional central extensions of naturally graded non-Lie filiform Leibniz algebras}

In this section we present the classification of $k$-dimensional central extensions of naturally graded non-Lie filiform Leibniz algebras. Let $V=<x_1,x_2,\dots,x_k>$ be an abelian algebra. First, we compute a basis of $HL^2(F_n^i,V)$ with $i=1,2$.

\begin{prop}\label{cocyclefil}
\begin{itemize}

\

\item The following cocycles
$$\theta_{i,j}(e_i,e_1)=x_j, \ \  1\leq i\leq n, \ \
\theta_{n+1,j}(e_1,e_2)=x_j, \ \ \theta_{n+2,j}(e_2,e_2)=x_j$$
form a basis of $ZL^2(F_n^1,V)$ and
$ZL^2(F_n^2,V);$
	
\item The following coboundaries
\begin{itemize}
\item[$a)$] $\theta_{1,j}(e_1,e_1)=x_j, \ \ \theta_{1,j}(e_2,e_1)=x_j,\ \
 \theta_{i,j}(e_{i+1},e_1)=x_j, \ \ 2\leq i\leq n-2, \ \
1\leq j\leq k,$ \\
form a basis of $BL^2(F_n^1,V)$,

\item[$b)$] $\theta_{1,j}(e_1,e_1)=x_j,\ \
\theta_{i,j}(e_{i+1},e_1)=x_j, \ \ 2\leq i\leq n-2, \ \ 1\leq
j\leq k,$ \\
form a basis of $BL^2(F_n^2,V)$;
\end{itemize}

 \item  Basis of the spaces $HL^2(F_n^1,V)$ and
 $HL^2(F_n^2,V)$ are formed by the following cocycles
$$\theta_{1,j}(e_2,e_1)=x_j, \ \ \theta_{2,j}(e_n,e_1)=x_j,\ \
\theta_{3,j}(e_{1},e_2)=x_j,  \ \ \theta_{4,j}(e_{2},e_2)=x_j,\ \
1\leq j\leq k.$$
\end{itemize}
\end{prop}
\begin{proof} The proof is carried out by straightforward calculations.
\end{proof}

\noindent {\bf One-dimensional central extensions.} In this subsection we study the classification of one-dimensional central extensions of the algebras $F_n^1$ and $F_n^2.$

\begin{thm}\label{one-dimeFn1}
An arbitrary one-dimensional central extension of the complex algebra $F_n^1$ is isomorphic to one of the following pairwise non-isomorphic algebras:
 $$F^1_{n}\oplus \mathbb{C}; \ \
F_{n+1}^1(0,\dots,0,\alpha_4,\alpha_3), \ \alpha_3,
\alpha_4\in\{0,1\}; \ \ L_{n+1}^{3,\lambda}; \ \
L_{n+1}^{4,\lambda};\ L_{n+1}^{5,\lambda,\mu};\ L_{n+1}^1.$$
\end{thm}

\begin{proof}
From Proposition \ref{cocyclefil} and the construction of central extensions, we derive that the table of multiplication of one-dimensional extension of $F_n^1$ has the form:
$$\begin{array}{llll}
[e_1,e_1]=e_{3}, & [e_i,e_1]=e_{i+1},& 3\leq i\leq n-1,&\\[1mm]
[e_2,e_1]=e_{3}+\alpha_1x, & [e_n,e_1]=\alpha_2x, &
[e_1,e_2]=\alpha_3x, & [e_2,e_2]=\alpha_4x.
\end{array}$$

We distinguish the following cases:

\

 \noindent $\bullet$ If $\alpha_i=0$ for any $1\leq i\leq4,$  then we get algebra $F_{n}^1\oplus \mathbb{C}.$

\

 \noindent $\bullet$  If there exists $i$ such that $\alpha_i\neq0,$ then we consider the cases.

\

\noindent \textbf{Case 1.} Let $\alpha_2\neq0,$ then one can assume $\alpha_2=1.$ Note that $dimL^i=n+1-i$ for $2\leq i\leq n+1$. Consequently we have $(n+1)$-dimensional filiform Leibniz algebra. Making the following change of basis $e_2^\prime=e_2-\alpha_1e_n$ and $e_{n+1}^\prime=x,$ we derive $\alpha_1=0.$ Hence, in this case we obtain the family of algebras $F_{n+1}^1(0,0,\dots,0,\alpha_4,\alpha_3)$.

In the work \cite{gomez-omirov}, the authors show that instead the general change of basis in the family $F_{n+1}^1(\alpha_4,\alpha_5, \dots, \alpha_{n+1}, \theta )$ it is sufficient to consider the following change of basis:
$$\begin{array}{lll}
e_1^\prime=Ae_1+Be_2,& e_2^\prime=(A+B)e_2+B(\alpha_3-\alpha_4)e_n,&
e_3^\prime=A(A+B)e_3+(AB\alpha_3+B^2\alpha_4)e_{n+1},\\
e_i^\prime=A^{i-2}(A+B)e_i, & 4\leq i\leq n+1,& A+B\neq 0.
\end{array}$$

From $[e_1^\prime,e_2^\prime]=\alpha_3^\prime e_{n+1}^\prime$
and $[e_2^\prime,e_2^\prime]=\alpha_4^\prime e_{n+1}^\prime$,
we obtain
$$\alpha_3^\prime=\frac{A\alpha_3+B\alpha_4}{A^{n-1}},
\ \ \alpha_4^\prime=\frac{(A+B)\alpha_4}{A^{n-1}}.$$

Consider subcases.

\begin{itemize}
\item[$a)$] $\alpha_4=0$. Then $\alpha^\prime_4=0$
and $\alpha_3^\prime=\frac{\alpha_3}{A^{n-2}}.$

 If $\alpha_3=0,$ then $\alpha_3^\prime=0.$ Otherwise, putting $A=\sqrt[n-2]{\alpha_3}$ we have that $\alpha_3^\prime=1.$
Therefore, we have $F_{n+1}^1(0,\dots,0,0,\alpha_3),$ where $\alpha_3\in\{0,1\}.$

\item[$b)$] $\alpha_4\neq0$. Putting $B=\frac{A^{n-1}-A\alpha_4}{\alpha_4},$ we obtain $\alpha^\prime_4=1$ and $\alpha_3^\prime=\frac{\alpha_3-\alpha_4+A^{n-2}}{A^{n-2}}.$

If $\alpha_3=\alpha_4,$ then $\alpha_3^\prime=1,$ otherwise, putting
$A=\sqrt[n-2]{\alpha_4-\alpha_3},$ we get $\alpha_3^\prime=0$. Therefore, we obtain the algebras
$F_{n+1}^1(0,\dots,0,1,\alpha_3),$ where $\alpha_3\in\{0,1\}.$

\

Thus, we have $F_{n+1}^1(0,\dots,0,\alpha_4,\alpha_3),$ where $\alpha_3, \alpha_4\in\{0,1\}.$
\end{itemize}

\

\noindent \textbf{Case 2.} Let $\alpha_2=0.$ Taking the change of basis
 $$e_1^\prime=e_1,\ e_i^\prime=e_{i+1},\ 2\leq i\leq n-1, \ e_n^\prime=e_2, \ e^\prime_{n+1}=x,$$ we get the algebra
$$\begin{array}{lll}
[e_i,e_1]=e_{i+1},& 1\leq i\leq n-2,&\\[1mm]
[e_n,e_1]=e_{2}+\alpha_1e_{n+1},& [e_1,e_n]=\alpha_3e_{n+1},&
[e_n,e_n]=\alpha_4e_{n+1}.
\end{array}$$
Note that this algebra is $(n+1)$-dimensional quasi-filiform Leibniz algebra.

\begin{itemize}
\item\textbf{Case 2.1.} $\alpha_1\neq0$.

If $\alpha_4=0,$ then we get the algebra which is isomorphic to $L_{n+1}^{3,\lambda}$.

If $\alpha_3=0, \ \alpha_4\neq0,$ then obtained algebra is isomorphic to $L_{n+1}^{4,\lambda}$.

If $\alpha_3\neq0, \ \alpha_4\neq0$, then derived algebra is isomorphic to $L_{n+1}^{5,\lambda,\mu}$.

\item \textbf{Case 2.2.} $\alpha_1=0.$ Take the change of basis
$$e_1^\prime=e_1+e_n, \ e_2^\prime=2e_2+(\alpha_3+\alpha_4)e_{n+1}, \
e_i^\prime=2e_i, \ 3\leq i\leq n, \ e_{n+1}^\prime=e_{n+1}.$$ Then we obtain the algebra
$$\begin{array}{lll}
[e_i,e_1]=e_{i+1},& 1\leq i\leq n-2,&\\[1mm]
[e_n,e_1]=e_{2}+(\alpha_4-\alpha_3)e_{n+1},&
[e_1,e_n]=2(\alpha_3+\alpha_4)e_{n+1},&
[e_n,e_n]=4\alpha_4e_{n+1}.
\end{array}$$
It is easy to see that if $\alpha_3\neq\alpha_4,$ then we are in the case of $\alpha_1\neq0.$ Therefore, we may assume that $\alpha_3=\alpha_4\neq0.$ Making the change of basis $$e_i^\prime=e_i, \ 1\leq i\leq n-1, \
e_n^\prime=e_1-e_n, \ e_{n+1}^\prime=-4\alpha_3e_{n+1},$$ we deduce the algebra $L_{n+1}^{1}.$
\end{itemize}

Since obtained algebras in the case 1 are filiform and the algebras of the case 2 are quasi-filiform, they are non isomorphic.
\end{proof}

The next theorem describes one-dimensional central extensions of the algebra $F_n^2$.
\begin{thm}
An arbitrary one-dimensional central extension of the algebra $F_n^2$ is isomorphic to one of the following pairwise non-isomorphic algebras: $$F^2_{n}\oplus \mathbb{C}; \ \ F_{n+1}^2; \
F_{n+1}^2(0,\dots,0,0,1); \ F_{n+1}^2(0,\dots,0,1,0); \
\ L_{n+1}^{1,\lambda};\ L_{n+1}^{2,\lambda};\
L_{n+1}^2.$$
\end{thm}

\begin{proof}
	From Proposition \ref{cocyclefil} and from the construction of the central extensions, we derive that the table of multiplication of one-dimensional extension of $F_n^2$ has the form:
$$\begin{array}{llll}
[e_1,e_1]=e_{3},& [e_i,e_1]=e_{i+1},& 3\leq i\leq n-1,&\\[1mm]
[e_2,e_1]=\alpha_1x,&[e_n,e_1]=\alpha_2x,& [e_1,e_2]=\alpha_3x, &
[e_2,e_2]=\alpha_4x.
\end{array}$$

If $\alpha_i=0$ for any $1\leq i\leq4,$  then we get split algebra $F_{n}^2\oplus \mathbb{C}.$

Let $\alpha_i\neq 0$ for some $i$. Then we consider the possible cases.

\noindent\textbf{Case 1.} Let $\alpha_2\neq0.$ Scaling of basis elements we can assume
$\alpha_2=1.$ Since $dimL^i=n+1-i$ for $2\leq i\leq n+1$, we have $(n+1)$-dimensional filiform Leibniz algebra. Making the change $e_2^\prime=e_2-\alpha_1e_n$ and $e_{n+1}^\prime=x,$ we obtain
$\alpha_1=0.$ So, in this case we get the family of algebras $F_{n+1}^2(0,0,\dots,0,\alpha_3,\alpha_4).$

Arguing similarly as in the proof of Theorem \ref{one-dimeFn1} we obtain
the following pairwise non-isomorphic algebras:
$$F_{n+1}^2, \ \ F_{n+1}^2(0,\dots,0,1,0), \ \
F_{n+1}^2(0,\dots,0,0,1).$$

%
%
%

\

\noindent\textbf{Case 2.} Let $\alpha_2=0.$ Making the following change of basis:
$$e_1^\prime=e_1,\ e_i^\prime=e_{i+1},\ 2\leq i\leq n-1, \
e_n^\prime=e_2, \ e^\prime_{n+1}=x,$$ we get the family of algebras
$$\begin{array}{lll}
[e_i,e_1]=e_{i+1},& 1\leq i\leq n-2,&\\[1mm]
[e_n,e_1]=\alpha_1e_{n+1},& [e_1,e_n]=\alpha_3e_{n+1},&
[e_n,e_n]=\alpha_4e_{n+1}.
\end{array}$$

Note that these algebras are $(n+1)$-dimensional quasi-filiform Leibniz algebras of type I.

\begin{itemize}
\item \textbf{Case 2.1.} Let $\alpha_1\neq0.$ Then by rescaling basis elements we can assume $\alpha_1=1.$

If $\alpha_4=0,$ then obtained algebra is isomorphic to the algebra $L_{n+1}^{1,\lambda}$.

If $\alpha_4\neq0,$ then derived algebra is isomorphic to the algebra $L_{n+1}^{2,\lambda}$.

\item \textbf{Case 2.2.} Let $\alpha_1=0.$ Let us take the change of basis
$$e_1^\prime=e_1+e_n, \ e_2^\prime=e_2+(\alpha_3+\alpha_4)e_{n+1}, \
e_i^\prime=e_i, \ 3\leq i\leq n+1.$$ Then we obtain the family of algebras:
$$\begin{array}{lll}
[e_i,e_1]=e_{i+1},& 1\leq i\leq n-2,&\\[1mm]
[e_n,e_1]=\alpha_4e_{n+1},& [e_1,e_n]=(\alpha_3+\alpha_4)e_{n+1},&
[e_n,e_n]=\alpha_4e_{n+1}.
\end{array}$$

It is easy to see that if $\alpha_4\neq0,$ then we are in conditions of the case $\alpha_1\neq0.$ Therefore, we assume that $\alpha_4=0$ and $\alpha_3\neq0.$ Then making the change of basis $e_{n+1}^\prime=\alpha_3e_{n+1},$ we get $L_{n+1}^2$.
\end{itemize}
\end{proof}

\noindent {\bf Two-dimensional central extensions.}
This subsection is devoted to the study of classification of two-dimensional non-split central extensions of the algebras $F_n^1$ and $F_n^2.$

First, we consider two-dimensional central extensions of the algebra $F_n^1.$ From Proposition \ref{cocyclefil} and from the definition of central extensions, we conclude that the table of multiplication of two-dimensional central extension of the algebra $F_n^1$ has the form:
$$\begin{array}{ll}
[e_1,e_1]=e_{3},& [e_2,e_1]=e_{3}+\alpha_1x_1+\beta_1x_2,\\[1mm]
[e_i,e_1]=e_{i+1},& 3\leq i\leq n-1,\\[1mm]
[e_n,e_1]=\alpha_2x_1+\beta_2x_2,&
[e_1,e_2]=\alpha_3x_1+\beta_3x_2, \
[e_2,e_2]=\alpha_4x_1+\beta_4x_2.
\end{array}$$

Let us take the following change of basis
$$e_1^\prime=e_1,\ e_2^\prime=e_3,\
e_i^\prime=e_{i+1},\ 2\leq i\leq n-1, \ e_n^\prime=e_2, \
e_{n+1}=x_1, \ e_{n+2}=x_2.$$
Then we obtain:
$$\left\{\begin{array}{ll}
[e_i,e_1]=e_{i+1},& 1\leq i\leq n-2,\\[1mm]
[e_n,e_1]=e_{2}+\alpha_1e_{n+1}+\beta_1e_{n+2},&\\[1mm]
[e_{n-1},e_1]=\alpha_2e_{n+1}+\beta_2e_{n+2},&\\[1mm]
[e_1,e_n]=\alpha_3e_{n+1}+\beta_3e_{n+2},&\\[1mm]
[e_n,e_n]=\alpha_4e_{n+1}+\beta_4e_{n+2}.&
\end{array}\right.$$

\begin{thm}\label{2-dimensionalFn1}
An arbitrary two-dimensional non-split central extension of the algebra $F_n^1$ is isomorphic to one of the following pairwise non-isomorphic algebras:
	\begin{itemize}
	\item $	L(1,0,0,0),$ $ L(1,1,0,0),$ $ L(0,1,0,0),$ $ L(0,1,1,0),$ $
	L(1,0,0,\beta_4),\ \beta_4\neq 0,$ $ L(0,1,-1,0),$ $L(0,1,1,1),$ $L(0,1,2,4),$
	
	\noindent where
	$$L(\alpha_3,\alpha_4,\beta_3,\beta_4)\left\{\begin{array}{ll}
	[e_i,e_1]=e_{i+1},& 1\leq i\leq n-1,\\[1mm]
	[e_{n+1},e_1]=e_{2}+e_{n+2},&\\[1mm]
	[e_1,e_{n+1}]=\alpha_3e_{n}+\beta_3e_{n+2}, &\\[1mm]
	[e_{n+1},e_{n+1}]=\alpha_4e_{n}+\beta_4e_{n+2},
	\end{array}\right.$$
	\item $\begin{array}{lllllll}
	M(1,\beta_4),\ \beta_4\neq 1, & M(0,0), & M(0,1), & M(0,2), & M(2,1), & M(2,2), & M(1,1),
	\end{array}$
	
	\noindent where
	$$M(\alpha_4,\beta_4):\left\{\begin{array}{ll}
	[e_i,e_1]=e_{i+1},& 1\leq i\leq n-2,\\[1mm]
	[e_n,e_1]=e_{2}+e_{n+1},&\\[1mm]
	[e_1,e_n]=e_{n+2}, &\\[1mm]
	[e_n,e_n]=\alpha_4e_{n+1}+\beta_4e_{n+2},&
	\end{array}\right.$$
	
\item $L_{n+2}^1,$ $L_{n+2}^{3,\lambda},$ $L_{n+2}^{4,\lambda},\ L_{n+2}^{5,\lambda,\mu}$ and the algebra
$$ L^{*}:\left\{\begin{array}{ll}
[e_i,e_1]=e_{i+1},& 1\leq i\leq n-1,\\[1mm]
[e_1,e_{n+1}]=e_{2}+e_{n+2},& \\[1mm]
[e_i,e_{n+1}]=e_{i+1},& 2\leq i\leq n-1,\\[1mm]
[e_{n+1},e_{n+1}]=e_{n}.&
\end{array}\right.$$

\end{itemize}
\end{thm}

\begin{proof}

We distinguish the following cases:

\

\noindent \textbf{Case 1.}  Let $(\alpha_2,\beta_2)\neq(0,0)$. Without loss of generality, one can assume $\alpha_2\neq0$ and $[e_{n-1},e_1]=e_{n+1}.$ Making the change $e_n^\prime=e_n-\alpha_1e_{n-1}$, we obtain $\alpha_1=0.$ Setting
$e_n^\prime=e_{n+1},\ e_{n+1}^\prime=e_{n},$ we obtain the table of multiplication:
$$\begin{array}{ll}
[e_i,e_1]=e_{i+1},& 1\leq i\leq n-1,\\[1mm]
[e_{n+1},e_1]=e_{2}+\beta_1e_{n+2},&\\[1mm]
[e_1,e_{n+1}]=\alpha_3e_{n}+\beta_3e_{n+2}, &\\[1mm]
[e_{n+1},e_{n+1}]=\alpha_4e_{n}+\beta_4e_{n+2}.&
\end{array}$$

\textbf{Case 1.1.} Let $\beta_i=0$ for $1\leq i\leq 4.$ Then we  get a split central extension.

\

\textbf{Case 1.2.} Let $\beta_1\neq0.$ Putting $e_{n+2}'=\beta_1e_{n+2}$, we can assume that $\beta_1=1.$

Since $dimL^2=n$ and $dim L^i=n+1-i$ for $3\leq i \leq n$, we have $(n+2)$-dimensional quasi-filiform Leibniz algebra with characteristic sequence equal to $C(L)=(n,2).$

Let us take the general change of basis:
$$e_1^\prime=\sum\limits_{i=1}^{n+2}A_ie_i, \quad e_{n+1}^\prime=\sum\limits_{i=1}^{n+2}B_ie_i.$$

Then we express new basis elements $\{e_1^\prime, e_2^\prime,\dots,
e_{n+2}^\prime\}$ via basis elements $\{e_1, e_2,\dots, e_{n+2}\}.$ Then for $e_2^\prime$ we have
$$e_2^\prime=[e_1^\prime,e_1^\prime]=A_1(A_1+A_{n+1})e_2+A_1
\sum\limits_{i=2}^{n-2}A_{i}e_{i+1}+$$$$+
(A_{1}A_{n-1}+A_{1}A_{n+1}\alpha_3+A_{n+1}^2\alpha_4)e_n+
(A_{1}A_{n+1}+A_{1}A_{n+1}\beta_3+A_{n+1}^2\beta_4)e_{n+2}.$$
Further, we obtain the expressions for $e_t^\prime$ with $3\leq t\leq n,$ namely:
$$e_t^\prime=A_1^{t-1}(A_1+A_{n+1})e_t+A_1^{t-1}
\sum\limits_{i=2}^{n-t+1}A_{i}e_{i+t-1}, \ \ 3\leq t\leq n-1, \ \
\ e_n^\prime=A_1^{n-1}(A_1+A_{n+1})e_n.$$

Therefore, $A_1(A_1+A_{n+1})\neq0.$

From equations
$[e_{n-1}^\prime,e_{n+1}^\prime]=A_1^{n-2}(A_1+A_{n+1})B_1e_{n}=0$ we derive $B_1=0.$

Since $e_{n+2}^\prime=[e_{n+1}^\prime,e_1^\prime]-e_{2}^\prime$,
we deduce
$$e_{n+2}^\prime=(A_1B_{n+1}-A_1(A_1+A_{n+1}))e_2+
A_1\sum\limits_{i=2}^{n-2}(B_i-A_i)e_{i+1}+[(A_1B_{n-1}+A_{n+1}
B_{n+1}\alpha_4)-$$
$$(A_1A_{n-1}+A_1A_{n+1}\alpha_3+A_{n+1}^2\alpha_4)]e_{n}+[(A_1B_{n+1}+A_{n+1}
B_{n+1}\beta_4)-(A_1A_{n+1}+A_1A_{n+1}\beta_3+A_{n+1}^2\beta_4)]e_{n+2}.$$

Taking into account that in new basis $[e_{n+2}^\prime,e_{1}^\prime]=0,$
we derive restrictions:
$$B_{n+1}=A_1+A_{n+1}, \ \  B_i=A_i, \ \ 2\leq i\leq n-2.$$

Considering the products
$$[e_1^\prime,e_{n+1}^\prime]=\alpha_3^\prime
e_{n}^\prime+\beta_3^\prime e_{n+2}^\prime, \ \
[e_{n+1}^\prime,e_{n+1}^\prime]=\alpha_4^\prime e_{n}^\prime
+\beta_4^\prime e_{n+2}^\prime,$$ we derive expressions for the new parameters:
$$\begin{array}{lr}
\alpha_3^\prime=\displaystyle\frac{A_1^2\alpha_3+A_1A_{n+1}\alpha_4+
A_{n+1}B_{n+1}(\alpha_3\beta_4-\alpha_4\beta_3)+
A_1(A_{n-1}-B_{n-1})\beta_3+
A_{n+1}(A_{n-1}-B_{n-1})\beta_4}{A_1^{n-1}(A_1+A_{n+1}\beta_4-A_{n+1}\beta_3)},&(1)\\[5mm]
\alpha_4^\prime=\displaystyle\frac{B_{n+1}(A_1\alpha_4+
A_{n+1}(\alpha_3\beta_4-\alpha_4\beta_3)+(A_{n-1}-B_{n-1})\beta_4)}{A_1^{n-1}(A_1+A_{n+1}\beta_4-A_{n+1}\beta_3)},&(2)\\[5mm]
\beta_3^\prime=\displaystyle\frac{(A_1\beta_3+A_{n+1}\beta_4)B_{n+1}}{A_1(A_1+A_{n+1}\beta_4-A_{n+1}\beta_3)},&(3)\\[5mm]
\beta_4^\prime=\displaystyle\frac{\beta_4B_{n+1}^2}{A_1(A_1+A_{n+1}\beta_4-A_{n+1}\beta_3)},&(4)
\end{array}$$
with condition
$$A_1(A_1+A_{n+1})(A_1+A_{n+1}\beta_4-A_{n+1}\beta_3)\neq0.$$

Consider now the possible cases.

\

\textbf{Case 1.2.1.} Let $\beta_4=0$. Then $\beta'_4=0$ and we obtain that $\beta_3$ is nullity invariant, that is, $\beta_3'=0$ if and only if $\beta_3=0$. Thus, we can distinguish two cases.

\begin{itemize}
\item[$a)$] Let $\beta_3=0$.
\begin{itemize}
\item[$a.1)$] $\alpha_4=0$. Then
$\alpha_3^\prime=\frac{\alpha_3}{A_1^{n-2}}, \ \
\alpha_4^\prime=0.$

If $\alpha_3=0,$ then $\alpha_3^\prime=0.$ Otherwise,
putting $A_1=\sqrt[n-2]{\alpha_3}$. Hence, we obtain
$\alpha_3^\prime=1$. In this subcase we get the algebras $L_{n+2}^{3,0}$ and $L(1, 0, 0,0).$

\item[$a.2)$] $\alpha_4\neq0.$ Then
$\alpha_3^\prime=\frac{A_1\alpha_3+A_{n+1}\alpha_4}{A_1^{n-1}}, \
\ \alpha_4^\prime=\frac{(A_1+A_{n+1})\alpha_4}{A_1^{n-1}}.$ Note that the following expression $\alpha_3^\prime-\alpha_4^\prime=
\frac{\alpha_3-\alpha_4}{A_1^{n-2}}$ give us one more nullity invariant.

If $\alpha_3-\alpha_4=0,$ then putting $A_{n+1}=
\frac{A_1^{n-1}-A_1\alpha_4}{\alpha_4}$, we have $\alpha_3^\prime=\alpha_4^\prime=1$ and the algebra $L(1,1,0,0).$

If $\alpha_3-\alpha_4\neq0,$ then choosing $A_{n+1}=-\frac{A_1\alpha_3}{\alpha_4}$ and
$A_1=\sqrt[n-2]{\alpha_4-\alpha_3}$, we obtain $\alpha_3^\prime=0,
 \ \alpha_4^\prime=1$. Therefore, we get the algebra $L(0,1,0,0).$
\end{itemize}

\item[$b)$] Let $\beta_3\neq0$. Then putting
$$A_{n-1}-B_{n-1}=-\frac{A_1^2\alpha_3+A_1A_{n+1}\alpha_4-
A_{n+1}B_{n+1}\alpha_4\beta_3}{A_1\beta_3}$$ we obtain
$$\alpha_3^\prime=0, \ \
\alpha_4^\prime=\frac{(A_1+A_{n+1})\alpha_4}{A_1^{n-1}}, \ \
\beta_3^\prime=\frac{(A_1+A_{n+1})\beta_3}{A_1-A_{n+1}\beta_3}.$$
From the last expression we deduce nullity invariant:
$$\beta_3^\prime+1=\frac{A_1(\beta_3+1)}{A_1-A_{n+1}\beta_3}.$$

We distinguish the following subcases:

\begin{itemize}
\item[$b.1)$] $\beta_3\neq-1.$ Putting
$A_{n+1}=\frac{A_1(1-\beta_3)}{2\beta_3}$, we obtain
$\beta_3^\prime=1, \ \
\alpha_4^\prime=\frac{(\beta_3+1)\alpha_4}{2A_1^{n-2}\beta_3}$.

If $\alpha_4=0,$ then $\alpha_4^\prime=0$ and we have the algebra $L_{n+2}^{3,1}.$

If $\alpha_4\neq0,$ then choosing $A_1=\sqrt[n-2]{\frac{(\beta_3+1)\alpha_4}{2\beta_3}}$, we obtain
$\alpha_4^\prime=1$ and the algebra $L(0,1,1,0).$

\item[$b.2)$] $\beta_3=-1.$ It implies $\beta_3^\prime=-1, \ \
\alpha_4^\prime=\frac{(A_1+A_{n+1})\alpha_4}{A_1^{n-1}}$.

If $\alpha_4=0,$ then $\alpha_4^\prime=0$ and the algebra
$L_{n+2}^{3,-1}$ is obtained.

If $\alpha_4\neq0,$ then setting $A_{n+1}=\frac{A_1^{n-1}-A_1\alpha_4}{\alpha_4}$, we get $\alpha_4^\prime=1$ and the algebra $L(0,1,-1,0).$
\end{itemize}
\end{itemize}

\

\textbf{Case 1.2.2.} Let $\beta_4\neq0$. Then similar as in naturally graded case (see the proof of Theorem 9 in \cite{JSC}) the values for $(\beta_3,\beta_4)$ are only following:  $(0,\beta_4),$ where $\beta_4\neq0,$ $(1,1)$ and $(2,4).$

\

\begin{itemize}
\item[$a)$] Let $(\beta_3,\beta_4)=(0,\beta_4)$ with $\beta_4\neq 0.$ Then from expressions $(1)-(4)$
we get
$$\alpha_3^\prime=\frac{\alpha_3}{A_1^{n-1}}, \ \
\alpha_4^\prime=\frac{A_1\alpha_4+(A_{n-1}-B_{n-1})\beta_4}{A_1^{n-1}},
\  \ \ A_{n+1}=0.$$

Taking $A_{n-1}-B_{n-1}=-\frac{A_1\alpha_4}{\beta_4}$, we obtain $\alpha_4^\prime=0$.

If $\alpha_3=0,$ then we have $\alpha_3^\prime=0$ and we obtain the algebra
$L_{n+2}^{4,\lambda}$ with $\lambda\neq0.$

If $\alpha_3\neq0,$ choosing $A_1=\sqrt[n-1]{\alpha_3}$,
we have $\alpha_3^\prime=1$ and the algebra $L(1,0,0,\beta_4),$ $\beta_4\neq0.$

\item[$b)$] Let $(\beta_3,\beta_4)=(1,1).$ Then expressions $(1)-(4)$ imply
$$\alpha_3^\prime=\frac{A_1\alpha_3+A_{n-1}-B_{n-1}}{A_1^{n-1}}, \ \
\alpha_4^\prime=\frac{A_1\alpha_4+(A_{n-1}-B_{n-1})}{A_1^{n-1}}, \
\ \ A_{n+1}=0.$$

Setting $A_{n-1}-B_{n-1}=-A_1\alpha_3$, we deduce $\alpha_3^\prime=0$ and $\alpha_4^\prime=\frac{\alpha_4-\alpha_3}{A_1^{n-2}}.$

Note that the following expression
$$\alpha_3^\prime-\alpha_4^\prime=\frac{\alpha_3-\alpha_4}{A_1^{n-2}}$$
is nullity invariant.

If $\alpha_3-\alpha_4=0,$ then we obtain $\alpha_4^\prime=0$ and the algebra $L_{n+2}^{5,1,1}.$

If $\alpha_3-\alpha_4\neq0,$ then choosing
$A_1=\sqrt[n-2]{\alpha_4-\alpha_3}$, we have $\alpha_4^\prime=1$ and the algebra $L(0,1,1,1).$

\item[$c)$] Let $(\beta_3,\beta_4)=(2,4).$ Then expressions $(1)-(4)$ transform to the following form:
$$\alpha_3^\prime=\frac{A_1\alpha_3+2(A_{n-1}-B_{n-1})}{A_1^{n-1}}, \ \
\alpha_4^\prime=\frac{A_1\alpha_4+4(A_{n-1}-B_{n-1})}{A_1^{n-1}},
\ \ \ A_{n+1}=0.$$

Putting $A_{n-1}-B_{n-1}=-\frac{A_1\alpha_3}{2}$, we derive
$\alpha_3^\prime=0$ and $\alpha_4^\prime=\frac{\alpha_4-2\alpha_3}{A_1^{n-2}}.$

It is easy to check that the following expression:
$$2\alpha_3^\prime-\alpha_4^\prime=\frac{2\alpha_3-\alpha_4}{A_1^{n-2}}$$
is nullity invariant.

If $\alpha_3-2\alpha_4=0,$ then we have $\alpha_4^\prime=0$ and the algebra $L_{n+2}^{5,2,4}.$

If $\alpha_3-2\alpha_4\neq0,$ then choosing $A_1=\sqrt[n-2]{\alpha_4-2\alpha_3}$, we obtain $\alpha_4^\prime=1$
and the algebra $L(0,1,2,4).$
\end{itemize}

\textbf{Case 1.3.} Let $\beta_1=0.$
Let us make the following change of basis:
$$e_1^\prime=e_1+(\alpha_4-\alpha_3)e_{n-1}+e_{n+1}, \
e_2^\prime=2e_2+2\alpha_4e_n+(\beta_3+\beta_4)e_{n+2}, \
e_i^\prime=2e_i, \ 3\leq i\leq n+1, \ e_{n+2}^\prime=e_{n+2}.$$
Then the table of multiplication of the algebra we have the form:
$$\begin{array}{lll}
[e_i,e_1]=e_{i+1},& 1\leq i\leq n-1,&\\[1mm]
[e_{n+1},e_1]=e_{2}+(\beta_4-\beta_3)e_{n+2},&
[e_1,e_{n+1}]=(\alpha_3+\alpha_4)e_{n}+2(\beta_3+\beta_4)e_{n+2},&\\[1mm]
[e_{n+1},e_{n+1}]=2\alpha_4e_{n}+4\beta_4e_{n+2}.&
\end{array}$$

 We have $\beta_3=\beta_4\neq0.$ Indeed, if $\beta_3\neq\beta_4,$ then we are in the case of  $\beta_1\neq0$.

Putting $e_{n+2}^\prime=(\alpha_3+\alpha_4)e_n+4\beta_3e_{n+2}$ we obtain an $(n+2)$-dimensional quasi-filiform Leibniz algebra with $C(L)=(n,1,1)$:
$$\begin{array}{lll}
[e_i,e_1]=e_{i+1},& 1\leq i\leq n-1,&\\[1mm]
[e_{n+1},e_1]=e_{2},& \\[1mm]
[e_1,e_{n+1}]=e_{n+2},&\\[1mm]
[e_{n+1},e_{n+1}]=\alpha_4e_{n}+e_{n+2}.&
\end{array}$$

Taking the general change of generator basis elements and applying similar arguments as in {\bf Case 1.2} we derive the expression for the new parameter:
$$\alpha_4^\prime=\frac{\alpha_4}{A_1^{n-2}}.$$
Since $\alpha_4$ is nullity invariant, we can easily get that $\alpha_4$ is equal either 0 or 1. Therefore, we have algebras:
$$\left\{\begin{array}{ll}
[e_i,e_1]=e_{i+1},& 1\leq i\leq n-1, \\[1mm]
[e_{n+1},e_1]=e_{2},& \\[1mm]
[e_{n+1},e_{n+1}]=e_{n+2},&\\[1mm]
[e_1,e_{n+1}]=e_{n+2},&
\end{array}\right.
\ \ \ \left\{\begin{array}{ll}
[e_i,e_1]=e_{i+1},& 1\leq i\leq n-1, \\[1mm]
[e_{n+1},e_1]=e_{2},& \\[1mm]
[e_{n+1},e_{n+1}]=e_{n}+e_{n+2},&\\[1mm]
[e_1,e_{n+1}]=e_{n+2}.&
\end{array}\right.$$

The following change of basis:
$$e_i^\prime=e_i, \ 1\leq i\leq n, \ e_{n+1}^\prime=e_1-e_{n+1}, \  e_{n+2}^\prime=-e_{n+2}.$$
leads to the algebras $L_{n+2}^1$  and $L^*.$

\

\noindent\textbf{Case 2.} Let $(\alpha_2,\beta_2)=(0,0)$ and
$(\alpha_1,\beta_1)\neq(0,0)$. Without loss of generality, we might assume that $\alpha_1\neq0.$ Therefore, taking
$e_{n+1}'=e_2+\alpha_1e_{n+1}+\beta_1e_{n+2}$, we get $[e_{n},e_1]=e_2+e_{n+1}.$

Thus, we have the table of multiplications:
$$\begin{array}{ll}
[e_i,e_1]=e_{i+1},& 1\leq i\leq n-2,\\[1mm]
[e_n,e_1]=e_{2}+e_{n+1},&\\[1mm]
[e_1,e_n]=\alpha_3e_{n+1}+\beta_3e_{n+2},&\\[1mm]
[e_n,e_n]=\alpha_4e_{n+1}+\beta_4e_{n+2}.&
\end{array}$$

Note that $dimL^2=n, \ dimL^i=n+1-i, \ \ 3\leq i \leq n.$

\

\textbf{Case 2.1.} Let $\beta_3\neq0$. Putting $e_{n+2}'=\alpha_3e_{n+1}+\beta_3e_{n+2},$ we can assume $\alpha_3=0, \ \beta_3=1$. Then we have the family of algebras:
$$M(\alpha_4,\beta_4):\left\{\begin{array}{ll}
[e_i,e_1]=e_{i+1},& 1\leq i\leq n-2,\\[1mm]
[e_n,e_1]=e_{2}+e_{n+1},& \\[1mm]
[e_1,e_n]=e_{n+2}, &\\[1mm]
[e_n,e_n]=\alpha_4e_{n+1}+\beta_4e_{n+2}.&
\end{array}\right.$$

Similar as in {\bf Case 1.2} we get the expressions for the new parameters
$$\alpha_4^\prime=\frac{B_{n}^2\alpha_4}{A_1^2+A_1A_{n}\beta_4+A_1A_{n}\alpha_4+A_{n}^2\alpha_4},$$
$$\beta_4^\prime=\frac{B_{n}(A_1\beta_4+A_{n}\alpha_4)}{A_1^2+A_1A_{n}\beta_4+A_1A_{n}\alpha_4+A_{n}^2\alpha_4},$$
with restriction $$A_1(A_1+A_{n})(A_1^2+A_1A_{n}\beta_4+A_1A_{n}\alpha_4+A_{n}^2\alpha_4)\neq0.$$

Consider
\begin{itemize}
\item  Let $\alpha_4=0$. Then $\alpha_4^\prime=0$ and $\beta_4^\prime=\frac{(A_1+A_{n})\beta_4}{A_1+A_{n}\beta_4}.$

If $\beta_4=0$, then we obtain  the algebra $M(0,0).$

If $\beta_4\neq0,$ then we have $\beta_4^\prime=\frac{(A_1+A_{n})\beta_4}{A_1+A_{n}\beta_4}.$ Note that the expression $\beta_4^\prime-1=\frac{A_1(\beta_4-1)}{A_1+A_{n}\beta_4}$ imply nullity invariant. Therefore the assumption $\beta_4=1$ deduce the algebra $M(0,1)$ and in opposite assumption (that is, $\beta_4\neq1$), choosing $A_{n}=\frac{A_1(\beta_4-2)}{\beta_4},$ we obtain $\beta_4^\prime=2$ and the algebra $M(0,2).$

\item Let $\alpha_4\neq0$. The following expressions:
$$\alpha_4^\prime-\beta_4^\prime=\frac{A_1(A_1+A_n)(\alpha_4-\beta_4)}{A_1^2+A_1A_{n}\beta_4+A_1A_{n}\alpha_4+A_{n}^2\alpha_4},$$ $$ \beta_4^\prime-1=\frac{A_1^2(\beta_4-1)}{A_1^2+A_1A_{n}\beta_4+A_1A_{n}\alpha_4+A_{n}^2\alpha_4}$$
 give us nullity invariants.

\begin{itemize}
	\item[$a)$] $\alpha_4-\beta_4\neq0$ and
$\beta_4\neq1.$ Putting
$A_{n}=\frac{A_1(1-\alpha_4)}{\alpha_4-\beta_4}$, we obtain
 the algebra $M(1,\beta_4),$ where $\beta_4\neq1.$

\item[$b)$] $\alpha_4-\beta_4\neq0$ and
$\beta_4=1.$ Setting $A_{n}=\frac{(\alpha_4-2)A_1}{\alpha_4}$, we have
$\alpha_4^\prime=2$ and the algebra $M(2,1).$

\item[$c)$] $\alpha_4-\beta_4=0$ and
$\beta_4\neq1.$ Choosing $A_{n}=\frac{A_1(-\alpha_4+\sqrt{2\alpha_4^2-2\alpha_4})}{\alpha_4}$,
we get $\alpha_4^\prime=2$ and the algebra $M(2,2).$

\item[$d)$] $\alpha_4-\beta_4=0$ and $\beta_4=1.$
Then we obtain $\alpha_4^\prime=\beta_4^\prime=1$ and the algebra $M(1,1).$
\end{itemize}

\end{itemize}

\

\textbf{Case 2.2.} Let $\beta_3=0$ and $\beta_4\neq0$. Let us take the change of basis as follows:
$$e_1^\prime=A_1e_1+A_ne_n, \ \
e_{n}^\prime=(A_1+A_n)e_n, \ \ \mbox{with}  \ A_1(A_1+A_{n})\neq0.$$

Similar as in {\bf Case 1.2} we get
$e_{n+1}^\prime=[e_{n}^\prime,e_1^\prime]-e_{2}^\prime=(A_1^2+A_1A_{n}\alpha_4-A_1A_{n}\alpha_3)e_{n+1}+A_1A_{n}\beta_4e_{n+2}.$
On the other hand, we have
$[e_{1}^\prime,e_n^\prime]=(A_1+A_n)(A_{1}\alpha_3+A_{n}\alpha_4)e_{n+1}+(A_1+A_n)A_{n}\beta_4e_{n+2}.$

Consider the determinant:
$$\left|%
\begin{array}{cc}
  A_1^2+A_1A_{n}\alpha_4-A_1A_{n}\alpha_3 & A_1A_{n}\beta_4 \\[1mm]
  (A_1+A_n)(A_{1}\alpha_3+A_{n}\alpha_4)&(A_1+A_n)A_{n}\beta_4 \\
\end{array}%
\right|=A_1A_n(A_1+A_{n})\beta_4(A_1-A_{n}\alpha_3-A_1\alpha_3)\neq0.$$

Therefore, we conclude that $e'_{n+1}$ and $[e'_1,e'_n]$ are lineally independent. Therefore, we are in the conditions of {\bf Case 2.1}.

\

\textbf{Case 2.3.} Let $(\alpha_1, \beta_1)=(0,0).$ Taking the following change of basis:
$$e_1^\prime=e_1+e_{n}, \
e_2^\prime=2e_2+(\alpha_3+\alpha_4)e_{n+1}+(\beta_3+\beta_4)e_{n+2}, \
e_i^\prime=2e_i, \ 3\leq i\leq n, \ e_{n+1}^\prime=e_{n+1}, \
e_{n+2}^\prime=e_{n+2},$$ we obtain the family of algebras:
$$\begin{array}{lll}
[e_i,e_1]=e_{i+1},& 1\leq i\leq n-2,&\\[1mm]
[e_{n},e_1]=e_{2}+(\alpha_4-\alpha_3)e_{n+1}+(\beta_4-\beta_3)e_{n+2},&\\[1mm]
[e_1,e_{n}]=2(\alpha_3+\alpha_4)e_{n+1}+2(\beta_3+\beta_4)e_{n+2},&\\[1mm]
[e_{n},e_{n}]=4\alpha_4e_{n+1}+4\beta_4e_{n+2}.&
\end{array}$$

Since $(\alpha_1,\beta_1)=(0,0),$ we can assume that  $(\alpha_3, \beta_3)=(\alpha_4, \beta_4)$. Making the change $e_{n+1}^\prime=4\alpha_3e_{n+1}+4\beta_3e_{n+2},$ we obtain a split algebra.
\end{proof}

\

Now, we study the two-dimensional non-split central extensions of the algebra $F_n^2.$ From Proposition \ref{cocyclefil} and definition of central extension, we conclude that the table of multiplication of two-dimensional central extension of the algebra $F_n^2$ has the following form:
$$\begin{array}{ll}
[e_1,e_1]=e_{3},& [e_2,e_1]=\alpha_1x_1+\beta_1x_2,\\[1mm]
[e_i,e_1]=e_{i+1},& 3\leq i\leq n-1,\\[1mm]
[e_n,e_1]=\alpha_2x_1+\beta_2x_2,&
[e_1,e_2]=\alpha_3x_1+\beta_3x_2, \
[e_2,e_2]=\alpha_4x_1+\beta_4x_2.
\end{array}$$

It we take the change of basis
$$e_1^\prime=e_1,\ e_2^\prime=e_3,\
e_i^\prime=e_{i+1},\ 2\leq i\leq n-1, \ e_n^\prime=e_2, \
e_{n+1}^\prime=x_1, \ e_{n+2}^\prime=x_2,$$
then its multiplication have the following form:
$$\left\{\begin{array}{ll}
[e_i,e_1]=e_{i+1},& 1\leq i\leq n-2,\\[1mm]
[e_n,e_1]=\alpha_1e_{n+1}+\beta_1e_{n+2},&\\[1mm]
[e_{n-1},e_1]=\alpha_2e_{n+1}+\beta_2e_{n+2},&\\[1mm]
[e_1,e_n]=\alpha_3e_{n+1}+\beta_3e_{n+2},&\\[1mm]
[e_n,e_n]=\alpha_4e_{n+1}+\beta_4e_{n+2}.&
\end{array}\right.$$

\begin{thm} An arbitrary two-dimensional non-split central extension of the algebra $F_n^2$ is isomorphic to one of the following pairwise non-isomorphic algebras:
	
	\begin{itemize}
	\item $N(1,0,0,0),$ $N(0,1,\beta_3,0),\ \alpha_4\in\{0,1\},\ \beta_3\neq 0,$ $N(1,0,0,1),$ $N(0,1,0,0),$ $N(0,1,1,1),$
	
	\noindent where
	$$N(\alpha_3,\alpha_4,\beta_3,\beta_4):\left\{\begin{array}{ll}
	[e_i,e_1]=e_{i+1},& 1\leq i\leq n-1,\\[1mm]
	[e_{n+1},e_1]=e_{n+2},&\\[1mm]
	[e_1,e_{n+1}]=\alpha_3e_{n}+\beta_3e_{n+2}, &\\[1mm]
	[e_{n+1},e_{n+1}]=\alpha_4e_{n}+\beta_4e_{n+2},&
	\end{array}\right.$$
	\item $R(0,0,1,0),\ R(0,0,1,1),\ R(0,1,1,\beta_4),\ \beta_4\in \mathbb{C}, R(1,0,0,1) $
	
	\noindent where
	$$R(\alpha_3,\alpha_4,\beta_3,\beta_4):\left\{\begin{array}{lll}
	[e_i,e_1]=e_{i+1},& 1\leq i\leq n-2,&\\[1mm]
	[e_n,e_1]=e_{n+1},&\\[1mm]
	[e_1,e_n]=\alpha_3e_{n+1}+\beta_3e_{n+2}, &\\[1mm]
	[e_n,e_n]=\alpha_4e_{n+1}+\beta_4e_{n+2},&
	\end{array}\right.$$
	
	\item $L_{n+2}^{1,\lambda},$ $L_{n+2}^{2,\lambda},$ $L_{n+2}^2$ and the algebra
$$N^*:	\left\{\begin{array}{ll}
		[e_i,e_1]=e_{i+1},& 1\leq i\leq n-1, \\[1mm]
		[e_1,e_{n+1}]=e_{n+2},& \\[1mm]
		[e_{n+1},e_{n+1}]=e_{n}.&
		\end{array}\right.$$
\end{itemize}
\end{thm}

\begin{proof}
	The proof is similar to the proof of Theorem \ref{2-dimensionalFn1}.
\end{proof}

\

\noindent{\bf Three-dimensional central extensions.} In this subsection we study three-dimensional non-split central extensions of the algebras $F_n^1$ and $F_n^2.$

From Proposition \ref{cocyclefil} and definition of central extension, we conclude that the table of multiplication of three-dimensional central extension of the algebra $F_n^1$ has the form:
$$\begin{array}{ll}
[e_1,e_1]=e_{3},&\\[1mm]
 [e_2,e_1]=e_{3}+\alpha_1x_1+\beta_1x_2+\gamma_1x_3,&\\[1mm]
[e_i,e_1]=e_{i+1},& 3\leq i\leq n-1,\\[1mm]
[e_n,e_1]=\alpha_2x_1+\beta_2x_2+\gamma_2x_3,&\\[1mm]
[e_1,e_2]=\alpha_3x_1+\beta_3x_2+\gamma_3x_3,&\\[1mm]
[e_2,e_2]=\alpha_4x_1+\beta_4x_2+\gamma_4x_3.&
\end{array}$$

Making the change of basis
$$e_1^\prime=e_1, \
e_i^\prime=e_{i+1},\ 2\leq i\leq n-1, \ e_n^\prime=e_2, \
e_{n+j}^\prime=x_j, \ 1\leq j \leq 3,$$
we obtain
$$\left\{\begin{array}{ll}
[e_i,e_1]=e_{i+1},& 1\leq i\leq n-2,\\[1mm]
[e_n,e_1]=e_{2}+\alpha_1e_{n+1}+\beta_1e_{n+2}+\gamma_1e_{n+3},&\\[1mm]
[e_{n-1},e_1]=\alpha_2e_{n+1}+\beta_2e_{n+2}+\gamma_2e_{n+3},&\\[1mm]
[e_1,e_n]=\alpha_3e_{n+1}+\beta_3e_{n+2}+\gamma_3e_{n+3}, &\\[1mm]
[e_n,e_n]=\alpha_4e_{n+1}+\beta_4e_{n+2}+\gamma_4e_{n+3}.&
\end{array}\right.$$

\begin{thm}\label{3-dimensionalFilFn1} An arbitrary three-dimensional non-split central extension of the algebra $F_n^1$ is isomorphic to one of the following pairwise non-isomorphic algebras:

\begin{itemize}
	\item $P(0,0,0),$ $P(1,0,0),$ $P(0,0,1),$ $P(1,0,1),$ $P(0,0,2),$ $P(1,0,2),$
	$P(0,1,\gamma_4), \ \gamma_4\neq 1,$  $P(0,2,1),$ $P(0,2,2),$ $P(0,1,1),$
	where
	$$P(\alpha_4,\beta_4,\gamma_4):\left\{\begin{array}{ll}
	[e_i,e_1]=e_{i+1},& 1\leq i\leq n-1,\\[1mm]
	[e_{n+1},e_1]=e_{2}+e_{n+2},&\\[1mm]
	[e_1,e_{n+1}]=e_{n+3}, &\\[1mm]
	[e_{n+1},e_{n+1}]=\alpha_4e_{n}+\beta_4e_{n+2}+\gamma_4e_{n+3},&
	\end{array}\right.$$
	
	\item
	$$P^*:\left\{\begin{array}{lll}
	[e_i,e_1]=e_{i+1},& 1\leq i\leq n-2,&\\[1mm]
	[e_n,e_1]=e_{2}+e_{n+1},&\\[1mm]
	[e_1,e_n]=e_{n+2},&\\[1mm]
	[e_n,e_n]=e_{n+3}.&
	\end{array}\right.$$
\end{itemize}	
\end{thm}
\begin{proof} Consider the cases.

\

\noindent\textbf{Case 1.} $(\alpha_2,\beta_2,\gamma_2)\neq(0,0,0)$.
Without loss of generality, one can assume that $\alpha_2\neq0$ and $[e_{n-1},e_1]=e_{n+1}.$ Taking
$e_n^\prime=e_n-\alpha_1e_{n-1},$ we get $\alpha_1=0.$

Making the following change of basis:
$$e_i^\prime=e_{i},\ 1\leq i\leq n-1, \ e_n^\prime=e_{n+1},\ e_{n+1}^\prime=e_{n}, \
\ e_{n+2}^\prime=e_{n+2}, \ \ e_{n+3}^\prime=e_{n+3},$$
we transform the table of multiplication to the form:
$$\begin{array}{ll}
[e_i,e_1]=e_{i+1},& 1\leq i\leq n-1,\\[1mm]
[e_{n+1},e_1]=e_{2}+\beta_1e_{n+2}+\gamma_1e_{n+3},&\\[1mm]
[e_1,e_{n+1}]=\alpha_3e_{n}+\beta_3e_{n+2}+\gamma_3e_{n+3}, &\\[1mm]
[e_{n+1},e_{n+1}]=\alpha_4e_{n}+\beta_4e_{n+2}+\gamma_4e_{n+3}.&
\end{array}$$

If $\beta_i=0$ or $\gamma_i=0$ for any $i$, then we obtain a split Leibniz algebra. Therefore, we consider the case of $\beta_i\neq 0$ and $\gamma_j\neq0$ for some $i$ and $j.$

\

\textbf{Case 1.1.} $(\beta_1,\gamma_1)\neq(0,0).$ Without loss of generality, one can assume that $\beta_1\neq0$ and $[e_{n+1},e_1]=e_2+e_{n+2}.$ Hence, we have the following algebra:
$$\begin{array}{ll}
[e_i,e_1]=e_{i+1},& 1\leq i\leq n-1,\\[1mm]
[e_{n+1},e_1]=e_{2}+e_{n+2},&\\[1mm]
 [e_1,e_{n+1}]=\alpha_3e_{n}+\beta_3e_{n+2}+\gamma_3e_{n+3},&\\[1mm]
 [e_{n+1},e_{n+1}]=\alpha_4e_{n}+\beta_4e_{n+2}+\gamma_4e_{n+3},&
\end{array}$$
with $(\gamma_3,\gamma_4)\neq(0,0).$

\begin{itemize}
	
\item $\gamma_3\neq0.$ Then we may assume $[e_{1},e_{n+1}]=e_{n+3}$ (by replacing $e'_{n+3}=\alpha_3e_{n}+\beta_3e_{n+2}+\gamma_3e_{n+3}$). Thus, we have the family of algebras:
$$\begin{array}{ll}
[e_i,e_1]=e_{i+1},& 1\leq i\leq n-1,\\[1mm]
[e_{n+1},e_1]=e_{2}+e_{n+2},& \\[1mm]
[e_1,e_{n+1}]=e_{n+3},&\\[1mm]
[e_{n+1},e_{n+1}]=\alpha_4e_{n}+\beta_4e_{n+2}+\gamma_4e_{n+3}.&
\end{array}$$

Taking the general change of generator basis elements and applying similar arguments as in {\bf Case 1.2} of Theorem \ref{2-dimensionalFn1} we obtain expressions for the new parameters:
$$\begin{array}{lr}
\alpha_4^\prime=\displaystyle\frac{(A_1+A_{n+1})(A_1\alpha_4+(A_{n-1}-B_{n-1})\beta_4)}{A_1^{n-2}(A_1^2+A_1A_{n+1}\beta_4+A_1A_{n+1}\gamma_4+A_{n+1}^2\beta_4)},&\\[3mm]
\beta_4^\prime=\displaystyle\frac{(A_1+A_{n+1})^2\beta_4}{A_1^2+A_1A_{n+1}\beta_4+A_1A_{n+1}\gamma_4+A_{n+1}^2\beta_4},&\\[3mm]
\gamma_4^\prime=\displaystyle\frac{(A_1+A_{n+1})(A_1\gamma_4+A_{n+1}\beta_4)}
{A_1^2+A_1A_{n+1}\beta_4+A_1A_{n+1}\gamma_4+A_{n+1}^2\beta_4},&
\end{array}$$
with restriction $$A_1(A_1+A_{n+1})(A_1^2+A_1A_{n+1}\beta_4+A_1A_{n+1}\gamma_4+A_{n+1}^2\beta_4)\neq0.$$

Consider all possible cases.

\begin{itemize}
\item[$a)$] $\beta_4=0$. Then $\beta_4^\prime=0.$

\begin{itemize}
\item[$a.1)$] $\gamma_4=0$. Then $\alpha_4^\prime=\frac{(A_1+A_{n+1})\alpha_4}{A_1^{n-1}}.$

If $\alpha_4=0,$ then $\alpha_4^\prime=0,$ otherwise,
putting $A_{n+1}= \frac{A_1^{n-1}-A_1\alpha_4}{\alpha_4},$ we get
$\alpha_4^\prime=1,$ and the algebras
$P(0, 0, 0),$ $P(1, 0, 0).$

\item[$a.2)$] $\gamma_4\neq0.$ It is remarkable that the following expression $\gamma_4^\prime-1=\frac{A_1(\gamma_4-1)}{A_1+A_{n+1}\gamma_4}$ deduce nullity invariant.

If $\gamma_4=1,$ then
$\alpha_4^\prime=\frac{\alpha_4}{A_1^{n-2}}, \ \ \gamma_4^\prime=1.$

If $\alpha_4=0,$ then we have the algebra $P(0,0,1)$ and when $\alpha_4\neq0,$  putting $A_1=\sqrt[n-2]{\alpha_4}$,
we get the algebra $P(1,0,1).$

\

If  $\gamma_4\neq1,$ then choosing
$A_{n+1}= \frac{A_1(\gamma_4-2)}{\gamma_4}$, we deduce
$\gamma_4^\prime=2$ and $\alpha_4^\prime=\frac{2\alpha_4}{\gamma_4A_1^{n-2}}.$
In the case of $\alpha_4=0,$ we have the algebra $P(0,0,2).$ In the opposite case (that is  $\alpha_4\neq0$) choosing $A_1=\sqrt[n-2]{\frac{2\alpha_4}{\gamma_4}}$, we have
$\alpha_4^\prime=1$ and the algebra $P(1,0,2).$
\end{itemize}

\item[$b)$] $\beta_4\neq0$. Putting
$A_{n-1}-B_{n-1}=-\frac{A_1\alpha_4}{\beta_4}$ we deduce
$\alpha_4^\prime=0.$ The equalities
$$\begin{array}{l}
\beta_4^\prime-\gamma_4^\prime=\displaystyle\frac{A_1(A_1+A_{n+1})(\beta_4-\gamma_4)}{A_1^2+A_1A_{n+1}\beta_4+A_1A_{n+1}\gamma_4+A_{n+1}^2\beta_4},\\[3mm]
 \gamma_4^\prime-1=\displaystyle\frac{A_1^2(\gamma_4-1)}{A_1^2+A_1A_{n+1}\beta_4+A_1A_{n+1}\gamma_4+A_{n+1}^2\beta_4},
 \end{array}$$
add two more nullity invariants.

Consider the following subcases:
\begin{itemize}
\item[$b.1)$] $\beta_4-\gamma_4\neq0$ and
$\gamma_4\neq1.$ Choosing
$A_{n+1}=\frac{A_1(1-\beta_4)}{\beta_4-\gamma_4}$, we get
$\beta_4^\prime=1,$ $\gamma_4^\prime=\frac{\beta_4\gamma_4-\gamma_4^2+\beta_4-\beta_4^2}{\beta_4(1-\gamma_4)}$ and the algebra $P(0,1,\gamma_4),$ where $\gamma_4\neq1.$

\item[$b.2)$] $\beta_4-\gamma_4\neq0$ and
$\gamma_4=1.$ Setting $A_{n+1}=\frac{(\beta_4-2)A_1}{\beta_4}$ we derive
$\alpha_4^\prime=2, \ \gamma'_4=1$ and the algebra $P(0,2,1).$

\item[$b.3)$] $\beta_4-\gamma_4=0$ and $\gamma_4\neq1.$ Putting  $A_{n+1}=\frac{A_1(-\beta_4+\sqrt{2\beta_4^2-2\beta_4})}{\beta_4}$,
we obtain $\beta_4^\prime=2$ and the algebra $P(0,2,2).$

\item[$b.4)$] $\beta_4-\gamma_4=0$ and $\gamma_4=1.$ Then we get $\beta_4^\prime=\gamma_4^\prime=1$ and the algebra $P(0,1,1).$
\end{itemize}
\end{itemize}

\item $\gamma_3=0$. It implies $\gamma_4\neq0$.
Thus, we have the following family of algebras:
$$\begin{array}{ll}
[e_i,e_1]=e_{i+1},& 1\leq i\leq n-1,\\[1mm]
[e_{n+1},e_1]=e_{2}+e_{n+2},&\\[1mm]
[e_1,e_{n+1}]=\alpha_3e_{n}+\beta_3e_{n+2},&\\[1mm]
 [e_{n+1},e_{n+1}]=\alpha_4e_{n}+\beta_4e_{n+2}+\gamma_4e_{n+3}.&
\end{array}$$

Let us take the general change of generator basis elements:
$$e_1^\prime=A_1e_1+A_{n+1}e_{n+1}, \ \
e_{n+1}^\prime=(A_1+A_{n+1})e_{n+1}, \ \mbox{with} \ A_1(A_1+A_{n+1})\neq0.
$$

Let us take the general change of generator basis elements:
$$e_1^\prime=A_1e_1+A_{n+1}e_{n+1}, \ \
e_{n+1}^\prime=(A_1+A_{n+1})e_{n+1}.$$
Similar to above we derive:

$$e_2^\prime=[e_1^\prime,e_1^\prime]=A_1(A_1+A_{n+1})e_2+
(A_{1}A_{n+1}\alpha_3+A_{n+1}^2\alpha_4)e_{n}+$$
$$
(A_{1}A_{n+1}+A_{1}A_{n+1}\beta_3+A_{n+1}^2\beta_4)e_{n+2}+
A_{n+1}^2\gamma_4e_{n+3}$$ and
$$e_t^\prime=A_1^{t-1}(A_1+A_{n+1})e_t, \ \ 3\leq t\leq n.$$

Since $e_{n+2}^\prime=[e_{n+1}^\prime,e_1^\prime]-e_{2}^\prime$, then
$$e_{n+2}^\prime=(A_{1}A_{n+1}\alpha_4-A_1A_{n+1}\alpha_3)e_{n}+
(A_{1}^2+A_{1}A_{n+1}\beta_4-A_1A_{n+1}\beta_3)e_{n+2}+A_1A_{n+1}\gamma_4e_{n+3}.$$

We have
$$[e_{1}^\prime,e_{n+1}^\prime]=(A_1+A_{n+1})(A_{1}\alpha_3+A_{n+1}\alpha_4)e_{n}+
(A_1+A_{n+1})(A_{1}\beta_3+A_{n+1}\beta_4)e_{n+2}+(A_1+A_{n+1})A_{n+1}\gamma_4e_{n+3}.$$

Consider
$$\left|%
\begin{array}{cc}
  A_{1}^2+A_{1}A_{n+1}\beta_4-A_1A_{n+1}\beta_3 & A_1A_{n+1}\gamma_4 \\[1mm]
  (A_1+A_{n+1})(A_{1}\beta_3+A_{n+1}\beta_4)&(A_1+A_{n+1})A_{n+1}\gamma_4 \\
\end{array}%
\right|=A_1A_{n+1}(A_1+A_{n+1})\gamma_4(A_1-A_{n+1}\beta_3-A_1\beta_3)\neq0.$$

Therefore, we conclude that $e'_{n+2}$ and $[e'_1,e'_{n+1}]$ are lineally independent, which satisfy the conditions of the case $\gamma_3\neq0.$
\end{itemize}

\

\textbf{Case 1.2.} $(\beta_1,\gamma_1)=(0,0).$
Note that $\beta_3\gamma_4-\beta_4\gamma_3\neq0,$ because otherwise the algebra is split.

Let us make the following change of basis:
$$e_1^\prime=e_1+(\alpha_4-\alpha_3)e_{n-1}+e_{n+1}, \
e_2^\prime=2e_2+2\alpha_4e_n+(\beta_3+\beta_4)e_{n+2}+(\gamma_3+\gamma_4)e_{n+3},$$
$$e_i^\prime=2e_i, \ 3\leq i\leq n+1, \ e_{n+2}^\prime=e_{n+2}, \ e_{n+3}^\prime=e_{n+3}.$$
Then we obtain the following family of algebras:
$$\begin{array}{ll}
[e_i,e_1]=e_{i+1},& 1\leq i\leq n-1,\\[1mm]
[e_{n+1},e_1]=e_{2}+(\beta_4-\beta_3)e_{n+2}+(\gamma_4-\gamma_3)e_{n+3},&\\[1mm]
[e_1,e_{n+1}]=(\alpha_3+\alpha_4)e_{n}+2(\beta_3+\beta_4)e_{n+2}+2(\gamma_3+\gamma_4)e_{n+3},&\\[1mm]
[e_{n+1},e_{n+1}]=2\alpha_4e_{n}+4\beta_4e_{n+2}+4\gamma_4e_{n+3},&
\end{array}$$

The assumption $(\beta_1,\gamma_1)=(0,0)$ implies $\beta_3=\beta_4\neq0$ and $\gamma_3=\gamma_4\neq0.$ But this is a contradiction to restriction $\beta_3\gamma_4-\beta_4\gamma_3\neq0.$

\

\noindent\textbf{Case 2.} $(\alpha_2,\beta_2,\gamma_2)=(0,0,0)$ and
$det\left(%
\begin{array}{ccc}
  \alpha_1&\beta_1&\gamma_1 \\
   \alpha_3&\beta_3&\gamma_3 \\
   \alpha_4&\beta_4&\gamma_4 \\
\end{array}%
\right)\neq0$. Without loss of generality we can assume that
$\alpha_1\beta_3\gamma_4\neq0,$ from which we have the algebra $P^*.$
\end{proof}

Now, we present three-dimensional central extensions of the algebra $F_n^2.$

\begin{thm}
	Three-dimensional non-split central extension of the algebra $F_n^2$ is isomorphic to one of the following pairwise non-isomorphic algebras:
	\begin{itemize}
		\item $Q(0,0,0,1,0),$ $Q(1,0,0,1,0),$  $Q(0,0,0,1,1),$ $Q(1,0,0,1,1),$ $Q(0,0,1,1,\gamma_4),\ \gamma_4\in \mathbb{C};$ $Q(0,1,0,0,1),$

		where
				$$Q(\alpha_4,\beta_3,\beta_4,\gamma_3,\gamma_4):\left\{\begin{array}{ll}
		[e_i,e_1]=e_{i+1},& 1\leq i\leq n-1,\\[1mm]
		[e_{n+1},e_1]=e_{n+2},&\\[1mm]
		[e_1,e_{n+1}]=\beta_3 e_{n+2}+\gamma_3 e_{n+3}, &\\[1mm]
		[e_{n+1},e_{n+1}]=\alpha_4e_{n}+\beta_4e_{n+2}+\gamma_4e_{n+3},&
		\end{array}\right.$$
		
		\item
$$Q^*:\left\{\begin{array}{ll}
		[e_i,e_1]=e_{i+1},& 1\leq i\leq n-2,\\[1mm]
		[e_n,e_1]=e_{n+1},&\\[1mm]
		[e_1,e_n]=e_{n+2},&\\[1mm]
		[e_n,e_n]=e_{n+3}.&
		\end{array}\right.$$
	\end{itemize}

\end{thm}

\begin{proof}
	The proof is carried out similar to the proof of Theorem \ref{3-dimensionalFilFn1}.
\end{proof}

\noindent{\bf Four-dimensional central extensions.} This subsection is devoted to the classification of four-dimensional non-split central extensions of the algebras $F_n^1$ and $F_n^2.$

Analogously, the table of multiplication of three-dimensional central extension of the algebra $F_n^1$ has the form:
$$\begin{array}{ll}
[e_1,e_1]=e_{3},& \\[1mm]
[e_i,e_1]=e_{i+1},& 3\leq i\leq n-1,\\[1mm]
[e_2,e_1]=e_{3}+\alpha_1x_1+\beta_1x_2+\gamma_1x_3+\delta_1x_4,&\\[1mm]
[e_n,e_1]=\alpha_2x_1+\beta_2x_2+\gamma_2x_3+\delta_2x_4,&\\[1mm]
[e_1,e_2]=\alpha_3x_1+\beta_3x_2+\gamma_3x_3+\delta_3x_4,&\\[1mm]
[e_2,e_2]=\alpha_4x_1+\beta_4x_2+\gamma_4x_3+\delta_4x_4.&
\end{array}$$

Making the change of basis
$$e_1^\prime=e_1, \
e_i^\prime=e_{i+1},\ 2\leq i\leq n-1, \ e_n^\prime=e_2, \
e_{n+1}^\prime=x_1,\ e_{n+2}^\prime=x_2,\ e_{n+3}^\prime=x_3,\ e_{n+4}^\prime=x_4,$$
we deduce
$$\left\{\begin{array}{ll}
[e_i,e_1]=e_{i+1},& 1\leq i\leq n-2,\\[1mm]
[e_n,e_1]=e_{2}+\alpha_1e_{n+1}+\beta_1e_{n+2}+\gamma_1e_{n+3}+\delta_1e_{n+4},&\\[1mm]
[e_{n-1},e_1]=\alpha_2e_{n+1}+\beta_2e_{n+2}+\gamma_2e_{n+3}+\delta_2e_{n+4},&\\[1mm]
[e_1,e_n]=\alpha_3e_{n+1}+\beta_3e_{n+2}+\gamma_3e_{n+3}+\delta_3e_{n+4}, &\\[1mm]
[e_n,e_n]=\alpha_4e_{n+1}+\beta_4e_{n+2}+\gamma_4e_{n+3}+\delta_4e_{n+4}.&
\end{array}\right.$$

\begin{thm}\label{4-dimensionalFilFn1}
	Four-dimensional non-split central extension of the algebra $F_n^1$ is isomorphic to the following algebra:
$$\left\{\begin{array}{ll}
[e_i,e_1]=e_{i+1},& 1\leq i\leq n-2,\\[1mm]
[e_{n-1},e_1]=e_{n+2},&\\[1mm]
[e_n,e_1]=e_{2}+e_{n+1},&\\[1mm]
[e_1,e_n]=e_{n+3},&\\[1mm]
[e_n,e_n]=e_{n+4}.&
\end{array}\right.$$
\end{thm}

\begin{proof}
Note that $det\left(%
\begin{array}{cccc}
  \alpha_1&\beta_1&\gamma_1&\delta_1 \\
  \alpha_2&\beta_2&\gamma_2&\delta_2 \\
  \alpha_3&\beta_3&\gamma_3&\delta_3 \\
  \alpha_4&\beta_4&\gamma_4&\delta_4 \\
\end{array}%
\right)\neq0$. Otherwise we get a split Leibniz algebra.

Without loss of generality, one can assume
$\alpha_1\beta_2\gamma_3\delta_4\neq0.$ Thus, we obtain the algebra of the theorem.
\end{proof}

Similar to the case of $F_n^1,$ we obtain the following result:
\begin{thm}\label{4-dimensionalFilFn1}
	Four-dimensional non-split central extension of the algebra $F_n^2$ is isomorphic to the following algebra:
$$\left\{\begin{array}{ll}
[e_i,e_1]=e_{i+1},& 1\leq i\leq n-2,\\[1mm]
[e_{n-1},e_1]=e_{n+2},&\\[1mm]
[e_n,e_1]=e_{n+1},&\\[1mm]
[e_1,e_n]=e_{n+3},&\\[1mm]
[e_n,e_n]=e_{n+4}.&
\end{array}\right.$$
\end{thm}

\

 Finally, we present the description of $k$-dimensional ($k\geq 5$) central extensions of naturally graded filiform non-Lie Leibniz algebras. In fact, all of them are split.

 \begin{thm}
 A	$k$-dimensional ($k\geq 5$) central extension of the algebras $F_n^1$ and $F_n^2$ is a split algebra.
 \end{thm}

\section*{Acknowledgments}

The authors were supported by Ministerio de Econom\'ia y Competitividad (Spain), grant MTM2013-43687-P (European FEDER support included) and by V Plan Propio de Investigación of Sevilla University (VPPI-US). The last named author was also partially supported by the Grant No. 0828/GF4 of Ministry of Education and Science of the Republic of Kazakhstan.

\end{document}